\documentclass[a4paper,12pt]{amsart}
\usepackage{amssymb,amsxtra,paralist,mathrsfs,upgreek,xcolor}
\usepackage[normalem]{ulem}
\usepackage[colorlinks=true,linkcolor=magenta,citecolor=cyan,urlcolor=olive]{hyperref}

\title[Sharp $L^p$-Estimates for Wave Equation on $ax+b$ Groups]{Sharp $L^p$-Estimates for Wave Equation\\
on $ax+b$ Groups}
\author[Y. Wang and L. Yan]{Yunxiang Wang and Lixin Yan}
\address{Yunxiang Wang, Department of Mathematics, Sun Yat-sen University, Guangzhou, 510275, P. R. China}
\email{\href{mailto: wangyx677@mail.sysu.edu.cn}{wangyx677@mail.sysu.edu.cn}}
\address{Lixin Yan, Department of Mathematics, Sun Yat-sen University, Guangzhou, 510275, P. R. China}
\email{\href{mailto: mcsylx@mail.sysu.edu.cn}{mcsylx@mail.sysu.edu.cn}}
\date{\today}

\numberwithin{equation}{section}
\newtheorem{theorem}{Theorem}
\newtheorem{lemma}[theorem]{Lemma}

\newtheorem{proposition}[theorem]{Proposition}
\theoremstyle{definition}
\newtheorem{definition}[theorem]{Definition}
\theoremstyle{remark}

\numberwithin{theorem}{section}

\newcommand{\R}{\mathbb{R}}
\newcommand{\e}{\mathrm{e}}
\newcommand{\Id}{\mathrm{Id}}
\newcommand{\indicator}{\mathbf{1}}
\newcommand{\supp}{\mathrm{supp}}

\newcommand{\arcosh}{\mathrm{arcosh}}

\newcommand{\BMO}{\mathrm{BMO}}
\newcommand{\SR}{\mathscr{R}}
\newcommand{\dV}{\mathrm{d}V}
\newcommand{\bc}{\mathbf{c}}
\newcommand{\F}{\mathscr{F}}
\newcommand{\g}{\mathbf{g}}

\renewcommand{\Re}{\mathrm{Re}}
\renewcommand{\Im}{\mathrm{Im}}
\renewcommand{\SS}{\mathbb{S}}
\renewcommand{\i}{\mathrm{i}}
\renewcommand{\d}{\mathrm{d}}

\allowdisplaybreaks

\begin{document}

\begin{abstract}
Let $G$ be the group $\mathbb{R}_+\ltimes \mathbb{R}^n$ endowed with Riemannian symmetric space metric $d$ and the right Haar measure $\mathrm{d} \rho$ which is of $ax+b$ type, and $L$ be the positive definite left-invariant distinguished Laplacian on $G$. Let $u=u(t,\cdot)$ be the solution to $\partial_t^2u+Lu=0$ with initial conditions $u|_{t=0}=f_0$ and $\partial_tu|_{t=0}=f_1$. In this article we show that for a fixed $t\in\mathbb{R}\setminus\{0\}$ and every $1<p<\infty$, 
\begin{align*}
\|u(t,\cdot)\|_{L^p(\mathrm{d}\rho)}\leq C_p\,\Big( (1+|t|)^{2|1/p-1/2|}\|f_0\|_{L^p_{\alpha_0}(\mathrm{d}\rho)}+(1+|t|)\,\|f_1\|_{L^p_{\alpha_1}(\mathrm{d}\rho)}\Big) 
\end{align*}
if and only if
\begin{align*}
\alpha_0\geq n\left|{1\over p}-{1\over2}\right| \quad \mbox{and}
\quad \alpha_1\geq n\left|{1\over p}-{1\over2}\right|-1.
\end{align*} 
The polynomial-in-time growth $1+|t|$ multiplying $\|f_1\|_{L^p_{\alpha_1}(\mathrm{d}\rho)}$ in the above estimate is best possible. 
This endpoint result settles the conjecture raised by M\"{u}ller and Thiele [Studia Math. \textbf{179} (2007)].

\end{abstract}

\maketitle


\section{Introduction}\label{sec1}
Let $ G$ be the Lie group $\R_+\ltimes \R^n$ endowed with the product: for all $(x,y), (x',y')\in G$,
\begin{align*}
(x,y)\cdot (x',y')=(xx',y+xy').
\end{align*}
The group $G$ is called an $ax+b$ group. Clearly, $e=(1, 0)$ is the identity of $G$. We endow $G$ with the left-invariant Riemannian metric $\g=x^{-2}(\d x^2 +\d y^2)$, and denote by $d$ the corresponding metric. This coincides with the metric on the $(n+1)$-dimensional hyperbolic space. The right Haar measure on $G$ is given by
\begin{align*}
\d \rho(x,y)=x^{-1}\,\d x\d y.
\end{align*}
The space $(G, d, \d \rho)$ is a solvable Lie group with exponential volume growth. Throughout the article, we will use the right Haar measure $\d \rho$ in notions such as $L^p(G)$ for all $1\leq p< \infty$.

A basis for the Lie algebra of left-invariant vector fields is given by
\begin{align}\label{e1.2}
X=x{\partial\over \partial x},\quad Y_j=x{\partial\over \partial y_j},\quad 1\leq j\leq n.
\end{align}
We define the left-invariant distinguished Laplacian to be the second order differential operator
\begin{align}\label{e1.3}
L=-X^2-\sum_{j=1}^nY_j^2.
\end{align}
The distinguished Laplacian $L$ has a self-adjoint extension in $L^2(G)$ (\cite{NeSt}), and thus we can use spectral calculus to define the operators
\begin{align*}
\e^{\i t\sqrt{L}}m(\sqrt{L}),
\end{align*}
where the multiplier $m$ will lie in a suitable symbol class. 

Let $u=u(t,(x,y))$ be the solution to the Cauchy problem
\begin{align}\label{e1.4}
{\partial^2u\over\partial t^2}+Lu=0,\quad u|_{t=0}=f_0,\quad \left.{\partial u\over\partial t}\right|_{t=0}=f_1.
\end{align}
For $f_0,f_1$ in $L^2(G)$ the solution is given by
\begin{align*}
u(t,\cdot)=\cos (t\sqrt{L})(f_0)+{\sin(t\sqrt{L})\over \sqrt{L}}(f_1).
\end{align*} 
M\"{u}ller and Thiele \cite{MuTh} studied Sobolev estimates for solutions to the wave equation \eqref{e1.4} and showed that for $t\in\R$ and $1\leq p<\infty$,
\begin{align}\label{e1.6}
\|u(t,\cdot)\|_{L^p(G)}\leq C_p\left((1+|t|)^{2|1/p-1/2|}\|f_0\|_{L^p_{\alpha_0}(G)}+(1+|t|)\,\|f_1\|_{L^p_{\alpha_1}(G)}\right) 
\end{align}
provided that 
\begin{align*}
\alpha_0>n\left|{1\over p}-{1\over 2}\right|
\quad\mbox{and}\quad \alpha_1>n\left|{1\over p}-{1\over 2}\right|-1.
\end{align*}
Here the adapted Sobolev norm 
is given by 
\begin{align*}
\|\varphi\|_{L^p_\alpha(G)}:=\|(\Id+L)^{\alpha/2}\varphi\|_{L^p(G)}.
\end{align*} 
Interestingly, despite the exponential volume growth of the underlying space $(G,d,\d\rho)$, M\"{u}ller and Thiele obtained polynomial-in-time growth estimates \eqref{e1.6}. The growth rates in \eqref{e1.6} are dimension-free, which are best possible when $p=1$ (cf. \cite[Theorem 5.6]{AkKuRuZh}). Later on, M\"{u}ller and Vallarino \cite{MuVa} obtained the estimate \eqref{e1.6} for solutions to the wave equation on Damek--Ricci spaces (see e.g. \cite{ADY96} for their definition). 

In \cite[Remark 8.1]{MuTh}, M\"{u}ller and Thiele conjectured that when $1<p<\infty$, \eqref{e1.6} holds for the endpoint cases $\alpha_0=n|{1/p}-{1/2}|$ and $\alpha_1=n|{1/p}-{1/2}|-1$.
In this article, we resolve this conjecture in the affirmative by showing the following result.

\begin{theorem}\label{thm1.1}
Let $G$ be the group $\R_+\ltimes \R^n$ which is of $ax+b$ type, and $L$ be the distinguished left invariant Laplacian on $G$. Assume $t\in\R\setminus\{0\}$. Then
\begin{asparaenum}[\rm (i)]
\item For $1<p<\infty$, the solution $u=u(t,(x,y))$ to the wave equation \eqref{e1.4} satisfies 
\begin{align}\label{s1e5}
\|u(t,\cdot)\|_{L^p(G)}\leq C_p\,\Big((1+|t|)^{2|1/p-1/2|}\|f_0\|_{L^p_{\alpha_0}(G)}+(1+|t|)\,\|f_1\|_{L^p_{\alpha_1}(G)}\Big) 
\end{align} 
if and only if 
\begin{align*}
\alpha_0\geq n\left|{1\over p}-{1\over 2}\right|
\quad \mbox{and}\quad \alpha_1\geq n\left|{1\over p}-{1\over 2}\right|-1.
\end{align*}
Moreover, as $|t|\to\infty$, the polynomial-in-time growth $1+|t|$ multiplying $\|f_1\|_{L^p_{\alpha_1}(\mathrm{d}\rho)}$ in the above estimate is best possible.
\smallskip
\item
For $p=1$, we have	
\begin{align}\label{s1e6}
\|u(t,\cdot)\|_{L^1(G)}\leq C\,(1+|t|) \big(\|(\Id+L)^{\alpha_0/2}f_0\|_{H^1(G)}+\|(\Id+L)^{\alpha_1/2}f_1\|_{H^1(G)}\big).
\end{align}
if and only if 
\begin{align*}
\alpha_0\geq{n\over 2}\quad \mbox{and}\quad \alpha_1\geq{n\over 2}-1.
\end{align*}
Here $H^1(G)$ is the Hardy space on the metric measure space $(G,d,\d\rho)$, see Subsection \ref{subsec2.2} below for its definition.
\end{asparaenum}
\end{theorem}

The Euclidean counterpart to Theorem~\ref{thm1.1} is given by Miyachi \cite{Mi} and Peral \cite{Pe} independently. The exponent $2|1/p-1/2|$ of $1+|t|$ in {\rm(i)} of Theorem~\ref{thm1.1} is independent of $n$, in contrast to the corresponding estimate for the Euclidean case \cite{Mi,Pe}, where it agrees with $n|1/p-1/2|$.

To show the endpoint estimate for $\alpha_0=n|{1/p}-{1/2}|$ and $\alpha_1=n|{1/p}-{1/2}|-1$ of \eqref{s1e5} and \eqref{s1e6} in Theorem~\ref{thm1.1}, we use an approach inspired by \cite{Pe} and \cite{Ta2}. We will factorize the wave propagator $(\Id+L)^{-n/4}\e^{\i t\sqrt{L}}$ into an operator given by convolution with some weighted measure on the sphere of radius $|t|$, and a multiplier of order $0$. Although the group $G$ is of exponential growth, we can still get polynomial-in-time growths in the estimates in Theorem \ref{thm1.1}. This phenomenon is due to the fact that the $L^1$--$L^1$-norm of the operator given by convolution with the weighted measure on the sphere of radius $|t|$ mentioned above brings an additional exponential decay on $|t|$, which can \emph{perfectly} eliminate the exponential growth. The proof relies heavily on the explicit asymptotic expansions of the spherical functions $\varphi_\lambda(t)$ on $G$, the spherical Fourier transform of the normalized measure on the sphere of radius $|t|$ as well as its derivative $\varphi_\lambda'(t)$ on $t$.

To show the sharpness of the Sobolev indices in Theorem \ref{thm1.1}, we use the fact that for some $\sigma\in\R$ and some $m$ belonging to the symbol class of order $\sigma$, the kernel of the wave propagator $m(\sqrt{L})\,\e^{\i t\sqrt{L}}$ concentrates near the sphere of radius $|t|$. Notice that $m(\sqrt{L})\, \e^{\i t\sqrt{L}}(f)$ concentrates near the identity whenever $f$ is a function supported near the sphere of radius $|t|$. We then construct a family of such functions to obtain the sharpness of the Sobolev indices in Theorem \ref{thm1.1}.

Finally, for the assertion on the sharpness of the polynomial-in-time growth in Theorem \ref{thm1.1}, we use the unit speed propagation property as well as the uniqueness of the solution to show that for all $1\leq p\leq\infty$ and $\alpha\in\R$,
\begin{align*}
\left\|{\sin(t\sqrt{L})\over\sqrt{L}}\right\|_{L^p_\alpha(G)\to L^p(G)}\geq C|t|.
\end{align*}
\smallskip

The paper is organized as follows. In Section \ref{sec2} we present basic facts about the Fourier theory as well as the Hardy space $H^1(G)$ on $G$. In Section \ref{sec4} we derive explicit asymptotic expansions of the spherical function $\varphi_\lambda(t)$ and its derivative $\varphi_\lambda'(t)$ on $t$, which will be used in the proof of Theorem \ref{thm1.1}. In Section \ref{sec5} we will prove the sufficiency part of Theorem \ref{thm1.1}. In Section \ref{sec6} we will obtain the necessity part of Theorem \ref{thm1.1} as well as the assertion on sharpness of the polynomial-in-time growth.
\medskip


\section{Preliminaries}\label{sec2}
In this section we review the analysis on the $ax+b$ group $G=\R_+\ltimes \R^n$. For details, we refer the reader to \cite{Br, Da, Gr, Pe2, Va}.
\smallskip

\subsection{The hyperbolic space}
The space $G$ endowed with the left-invariant Riemannian metric given by
\begin{align}\label{riem}
\g=x^{-2}(\d x^2+\d y^2)
\end{align}
is the $(n+1)$-dimensional hyperbolic space. The gradient induced by the Riemannian metric $\g$ reads
\begin{align}\label{e2.1}
\nabla f(x,y)=x^2\left({\partial f\over \partial x}(x,y)\,{\partial\over \partial x}+ \sum_{j=1}^n{\partial f\over \partial y_j}(x,y)\,{\partial \over \partial y_j}\right).
\end{align}
The geodesic distance $d$ is a left-invariant distance on $G$. For all $(x,y)$ in $G$
\begin{align}\label{e2.2}
d((x,y),e)=\arcosh{x^2+1+|y|^2\over 2x},
\end{align}
where $e$ is the identity of $G$ and $|\cdot|$ is the $n$-dimensional Euclidean norm. For simplicity in the sequel we denote $R(x,y)=d((x,y),e)$ for $(x,y)\in G$. The measure $\dV$ induced by the Riemannian metric $\g$ is a left Haar measure on $G$ and is given by
\begin{align*}
\dV(x,y)=x^{-(n+1)}\,\d x\d y_1\dots\d y_n.
\end{align*}
In particular,
\begin{align}\label{e2.4}
\dV (x,y)=\delta(x,y)\,\d \rho(x,y),
\end{align}
where the modular function is given by 
\begin{eqnarray}\label{mm}
\delta(x,y)=x^{-n}.
\end{eqnarray}

Following \cite[p. 21 and p. 136]{Pe2} and \cite[Exercise 3.18]{Gr}, we can adopt polar coordinates 
\begin{align*}
(x,y)=(x(r,\omega),y(r,\omega)), \quad (r,\omega)\in [0,\infty)\times \SS^n,
\end{align*}
where $r=R(x,y)\geq 0$, $\omega=(\omega',\omega_{n+1})\in \SS^n\subset \R^{1+n}$,
\begin{align}\label{pol}
x(r,\omega)={1\over(\omega_{n+1}+\coth r)\, \sinh r}\quad \mbox{and} \quad y(r,\omega)={\omega'\over \omega_{n+1}+\coth r},
\end{align}
Under these coordinates, the right Haar measure becomes
\begin{align}\label{e2.5}
\d \rho(x,y)=\big((\omega_{n+1}+\coth r)\, \sinh r\big)^{-n}(\sinh r)^n\,\d r\d\omega.
\end{align}
We now introduce a simple lemma that will be repeatedly used throughout this article.
\begin{lemma}\label{lem2.1}
For $r>0$, we have
\begin{align*}
\int_{\SS^n}\big((\omega_{n+1}+\coth r)\, \sinh r \big)^{-n/2}\,\d\omega&+\int_{\SS^n}\Big|\partial_r\!\Big[\big((\omega_{n+1}+\coth r)\, \sinh r \big)^{-n/2}\Big]\Big|\,\d\omega\\[2pt]
& \leq C\times \begin{cases}
1&\mbox{if }0<r\leq 1,\\
r\,\e^{-nr/2}&\mbox{if }r>1.
\end{cases}
\end{align*}
\end{lemma}
\begin{proof}
Let us first consider the case $0<r\leq 1$. Note that
\begin{align*}
\big((\omega_{n+1}+\coth r)\,\sinh r\big)^{-n/2}=(\omega_{n+1}\sinh r+\cosh r)^{-n/2}\leq C
\end{align*}
for $0<r\leq 1$ and $\omega\in\SS^n$, so
\begin{align*}
\int_{\SS^n}\big((\omega_{n+1}+\coth r)\, \sinh r \big)^{-n/2}\,\d\omega\leq C.
\end{align*}
Also, since
\begin{align}\label{e2.8}
\partial_r\!\Big[\big((\omega_{n+1}+\coth r)\, \sinh r\big)^{-n/2}\Big]=-{n\over 2}(\omega_{n+1}\sinh r+\cosh r)^{-n/2-1}(\omega_{n+1}+\tanh r)\,\cosh r,
\end{align}
for $0<r\leq 1$ and $\omega\in \SS^n$, the right hand side is bounded, and
\begin{align*}
\int_{\SS^n}\Big|\partial_r\!\Big[\big((\omega_{n+1}+\coth r)\, \sinh r \big)^{-n/2}\Big]\Big|\,\d\omega\leq C.
\end{align*}
These estimates finish the proof of the case $0<r\leq 1$.

Now we suppose $r\geq 1$. Using \eqref{mm} and \cite[\S D.2]{Gr2} we can write
\begin{align*}
{}&\int_{\SS^n}\big((\omega_{n+1}+\coth r)\,\sinh r\big)^{-n/2}\,\d\omega\\
\leq{}&C\,\e^{-nr/2}\int_{-1}^1(\omega_{n+1}+\coth r)^{-n/2}(1-\omega_{n+1}^2)^{(n-2)/2}\,\d\omega_{n+1}
\end{align*}
and by \eqref{e2.8},
\begin{align*}
&\int_{\SS^n}\Big|\partial_r\!\Big[\big((\omega_{n+1}+\coth r)\, \sinh r \big)^{-n/2}\Big]\Big|\,\d\omega\\
={}&{n\over 2}\int_{\SS^n}\big((\omega_{n+1}+\coth r)\, \sinh r\big)^{-n/2-1}|\omega_{n+1}+\tanh r|\,\cosh r\,\d\omega\\
\leq{}&C\,\e^{-nr/2}\int_{\SS^n}(\omega_{n+1}+\coth r)^{-n/2-1}\big((\omega_{n+1}+\coth r)+|\tanh r-\coth r|\big)\,\d \omega\\
\leq{}&C\,\e^{-nr/2}\left(\int_{-1}^1(\omega_{n+1}+\coth r)^{-n/2}(1-\omega_{n+1}^2)^{(n-2)/2}\,\d\omega_{n+1}\right.\\
{}&\quad+\left.(\coth r-\tanh r)\int_{-1}^1(\omega_{n+1}+\coth r)^{-(n/2+1)}(1-\omega_{n+1}^2)^{(n-2)/2}\,\d\omega_{n+1}\right).
\end{align*}
By a simple calculation, for $r>1$,
\begin{align*}
\coth r-\tanh r\leq C\, \e^{-2r} 
\end{align*}
and 
\begin{align*}
\int_{-1}^1(\omega_{n+1}+\coth r)^{-\mu}(1-\omega_{n+1}^2)^{(n-2)/2}\,\d\omega_{n+1}\leq C\times \begin{cases}
r&\mbox{if }\mu=n/2,\\[3pt]
1+\e^{(2\mu-n)r}&\mbox{if }\mu\ne n/2.
\end{cases}
\end{align*}
These estimates yield the desired result for $r>1$.
\end{proof}

Recall that a function $f(x,y)$ is said to be radial on $G$, if it depends only on $R(x,y)$. For such a function $f$, we have the following integration formula, which is a consequence of \eqref{mm}--\eqref{e2.5} and Lemma \ref{lem2.1}.
\begin{lemma}\label{lem2.2}
For every radial function $f$ on $G$,
\begin{align*}
\int_G \delta^{1/2}(x,y)\,f(x,y)\,\d\rho(x,y)=\int_0^\infty f(r)\,A(r)\,\d r,
\end{align*}
where
\begin{align*}
A(r)\leq C\times\begin{cases}
r^n&\mbox{if }0<r\leq 1,\\
r\,\e^{nr/2}&\mbox{if }r>1.
\end{cases}
\end{align*}
\end{lemma}
\smallskip

\subsection{Harmonic analysis on \texorpdfstring{$G$}{G}}\label{sec3}
One can define the spherical Fourier transform of a radial function $f$ by
\begin{align}\label{e2.9}
\F f(\lambda)=\nu_n\int_0^\infty f(r)\, \varphi_\lambda(r)\,(\sinh r)^n\,\d r,
\end{align}
where $\nu_n$ is the surface area of the unit sphere $\SS^n$ and $\varphi_{\lambda}$ denotes the elementary spherical function given by
\begin{align}\label{e2.10}
\varphi_\lambda(t)={2^{n/2-1}\Gamma((n+1)/2)\over \sqrt{\uppi}\,\Gamma(n/2)} (\sinh t)^{1-n}\int_{-t}^t\e^{\i \lambda s}(\cosh t-\cosh s)^{n/2-1}\,\d s,
\end{align}
see \cite[(3.1.4)]{Br}. We also have the inverse formula of the above Fourier transform:
\begin{align*}
f(r)={2^{n-2}\Gamma((n+1)/2)\over\uppi^{(n+3)/2}}\int_0^\infty \F f(\lambda)\,\varphi_\lambda(r)\,|\bc(\lambda)|^{-2}\,\d\lambda,
\end{align*}
where $\bc$ denotes the Harish-Chandra function given by
\begin{align}\label{e2.11}
\bc(\lambda)={2^{n-1}\Gamma((n+1)/2)\over\sqrt{\uppi}}{\Gamma(\i\lambda)\over \Gamma(n/2+\i\lambda)},
\end{align}
see \cite[(3.1.3)]{Br}.
\smallskip

Now we introduce the spectral multipliers on $G$. The distinguished Laplacian $L$ given by \eqref{e1.3} has a special relationship with the (positive definite) Laplace--Beltrami operator $\Delta $ on $G$. Indeed, in view of \cite[p. 177]{Da}, $\Delta$ has a continuous spectrum $[n^2/4,\infty)$ on $L^2(G,\dV)$. If we let $\Delta_n$ be the shifted operator $\Delta-n^2/4$, it follows from \cite[Proposition 2]{As} that
\begin{align*}
L(f)(x,y)=\delta^{1/2}(x,y)\, \Delta_n(\delta^{-1/2} f)(x,y).
\end{align*}
This combined with the spectral theorem gives
\begin{align}\label{e3.3}
\psi(\sqrt{L})(f)(x,y)=\delta^{1/2}(x,y)\,\psi(\sqrt{\Delta_n})(\delta^{-1/2}f)(x,y).
\end{align}

Let $k_\psi$ and $\kappa_\psi$ denote the convolution kernels of $\psi(\sqrt{L})$ and $\psi(\sqrt{\Delta_n})$ respectively, i.e.,
\begin{align*}
\psi(\sqrt{L})(f)=f*k_\psi\quad\mbox{and}\quad \psi(\sqrt{\Delta_n})(f)=f*\kappa_\psi,
\end{align*}
where ``$*$'' denotes the convolution on $G$ defined by
\begin{align}\label{con}
f*k(g)=\int_G f(h)\,k(h^{-1} g)\,\delta(h)\,\d\rho(h).
\end{align}
\begin{lemma}\label{lem3.1}
Let $\psi$ be a Borel function on $\R_+$. Then $\kappa_\psi$ is radial and is given by
\begin{align*}
\kappa_\psi(x,y)={2^{n-2}\Gamma((n+1)/2)\over\uppi^{(n+3)/2}}\int_0^\infty \psi(\lambda)\,\varphi_\lambda(R(x,y))\,|\bc(\lambda)|^{-2}\,\d \lambda.
\end{align*}
\end{lemma}
\begin{proof}
See \cite[Section 4]{Br}.
\end{proof}

In view of Lemma \ref{lem3.1} and \eqref{e3.3},
\begin{align}
k_\psi(x,y)&=\delta^{1/2}(x,y)\,\kappa_\psi(x,y)\notag\\
&={2^{n-2}\Gamma((n+1)/2)\over\uppi^{(n+3)/2}}\delta^{1/2}(x,y)\int_0^\infty \psi(\lambda)\,\varphi_\lambda(R(x,y))\,|\bc(\lambda)|^{-2}\,\d \lambda.\label{e3.4}
\end{align}
\smallskip

\subsection{Hardy spaces on \texorpdfstring{$G$}{G}.}\label{subsec2.2} Let us first recall the definition of Calder\'{o}n-Zygmund sets on $G$ which appears in \cite{HeSt} and implicitly in \cite{GiSj}.
\begin{definition}[{\cite[Definition 2.1]{Va}}]\label{def2.3}
A \emph{Calder\'{o}n-Zygmund set} on $G$ is a set $R=[x\,\e^{-s},x\,\e^s]\times Q$, where $Q$ is a dyadic cube in $\R^n$ of side length $\ell$ such that $x,s,\ell>0$ satisfy the following \emph{admissible condition}:
\begin{align*}
\begin{array}{lll}
\e^2x\,s\leq \ell\leq \e^8x\,s & &\mbox{if }0<s\leq 1,\smallskip\\
\e^{2s}x\leq \ell\leq \e^{8s}x & &\mbox{if }s>1.
\end{array}
\end{align*}
Let $\SR$ denote the family of all Calder\'{o}n-Zygmund sets.
\end{definition}
In \cite{HeSt} Hebisch and Steger proved that the space $G$ satisfies the \emph{Calder\'{o}n-Zygmund property}. More precisely, they proved that there exist two uniform constants $c_0,c_1\geq 1$ such that the following properties hold.
\begin{asparaitem}
\item For every $R\in\SR$, there are an $(x_R,y_R)\in R$ and an $r_R$ such that
\begin{align}\label{e2.12}
B((x_R,y_R),r_R)\subset R\subset B((x_R,y_R),c_0r_R),
\end{align}
where $B((x,y),r)$ is the ball in $G$ centered at $(x,y)$ of radius $r$.
\item For every $R\in\SR$, its dilated set is defined by $R^*:=\{(x,y)\in G:d((x,y),R)<r_R\}$, which satisfies
\begin{align*}
\rho(R^*)\leq c_1\rho(R).
\end{align*}
\item For every $R\in\SR$ there exist mutually disjoint sets $R_1,\dots,R_k\in\SR$, with $k=2$ or $2^n$, such that $R=\bigsqcup_{i=1}^kR_i$ and $\rho(R_i)=\rho(R)/k$ for $i=1,\cdots,k$.
\end{asparaitem}

Now we are in a position to introduce the Hardy space $H^1(G)$ defined in \cite{Va}. By replacing balls with Calder\'{o}n-Zygmund sets in the classical definition of atoms, we say that a function $a$ is an \emph{$(1,\infty)$-atom} if it satisfies the following properties.
\begin{asparaenum}
\item $a$ is supported in a Calder\'{o}n-Zygmund set $R$.
\item $\|a\|_\infty\leq(\rho(R))^{-1}$.
\item $\int a\,\d\rho=0$.
\end{asparaenum}
Observe that an $(1,\infty)$-atom is in $L^1(G)$ and it is normalized in such a way that its $L^1(G)$-norm does not exceed $1$.
\begin{definition}[{\cite[Definition 2.2]{Va}}]\label{def2.4}
The Hardy space $H^1(G)$ is the space of all functions $h$ in $L^1(G)$ such that $h=\sum_j\lambda_ja_j$, where $a_j$ are $(1,\infty)$-atoms and $\lambda_j\in \mathbb{C}$ such that $\sum_j|\lambda_j|<\infty$. We denote by $\|h\|_{H^1(G)}$ the infimum of $\sum_j|\lambda_j|$ over such decompositions.
\end{definition}
We have the following singular integral theorem and multiplier theorem as well as the interpolation theorem.
\begin{lemma}[{\cite[Theorem 4.1]{Va}}]\label{lem2.5}
Let $T$ be a linear operator which is bounded on $L^2(G)$ and admits a locally integrable Schwartz kernel $K$ with respect to $\d\rho$ off the diagonal which satisfies the condition
\begin{align*}
\sup_{R\in\mathscr{R}}\sup_{h,h'\in R}\int_{(R^*)^c}|K(g,h)-K(g,h')|\,\d \rho(g)<\infty.
\end{align*}
Then $T$ extends to a bounded operator from $H^1(G)$ to $L^1(G)$.
\end{lemma}

\begin{lemma}[{\cite[Theorem 5.3]{LiVaYa}}]\label{lem2.7}
Denote by $\Sigma$ the closed strip $\{\zeta\in\mathbb{C}:\Re \zeta\in[0,1]\}$. Suppose that $\{T_\zeta\}_{\zeta\in\Sigma}$ is a family of uniformly bounded operators on $L^2(G)$ such that the map $\zeta\mapsto \int_G T_\zeta(f)\, g\,\d \rho$ is continuous on $\Sigma$ and analytic in the interior of $\Sigma$, whenever $f,g\in L^2(G)$. Moreover, assume that
\begin{align*}
\|T_\zeta f\|_{L^1(G)}\leq A_1\|f\|_{H^1(G)}\quad\mbox{for all }f\in L^2(G)\cap H^1(G) \mbox{ and }\Re\zeta=0,
\end{align*}
and
\begin{align*}
\|T_\zeta f\|_{L^2(G)}\leq A_2\|f\|_{L^2(G)}\quad\mbox{for all }f\in L^2(G)\mbox{ and }\Re\zeta=1,
\end{align*}
where $A_1$ and $A_2$ are positive constants independent of $\Im \zeta\in\R$. Then for every $\vartheta\in(0,1)$ the operator $T_\vartheta$ is bounded on $L^{q_\vartheta}$ with $1/q_\vartheta=1-\vartheta/2$ and
\begin{align*}
\|T_\vartheta f\|_{L^{q_\vartheta}}\leq A_1^{1-\vartheta} A_2^\vartheta\|f\|_{L^{q_\vartheta}}\quad\mbox{for all }f\in L^2(G)\cap L^{q_\vartheta}(G).
\end{align*}
\end{lemma}

The following lemma is needed in the proof of the sufficiency part of Theorem \ref{thm1.1}. Recall that the symbol class is defined by
\begin{align*}
S^\sigma=\left\{m\in C^\infty(\R):\|m\|_{S^\sigma,k}:=\sup_{\lambda\in\R}\,(1+\lambda^2)^{-{\sigma-k\over 2}}|m^{(k)}(\lambda)|\leq C_{\sigma,k}\text{ for all }k\in \mathbb{N}\right\}.
\end{align*}
\begin{lemma}\label{thm3.2}
Let $m$ be an even function with $m\in S^\sigma$.
\begin{asparaenum}[\rm (a)]
\item If $\sigma=0$, then $m(\sqrt{L})$ is bounded from $H^1(G)$ to $L^1(G)$.
\item If $\sigma=-1$, then $|\nabla m(\sqrt{L})\cdot|_{\mathbf{g}}$ is bounded from $H^1(G)$ to $L^1(G)$, where $\nabla$ is the gradient given by \eqref{e2.1} and $|\cdot|_{\mathbf{g}}$ is the norm induced by the Riemannian metric $\g$ given by \eqref{riem}.
\end{asparaenum}
\end{lemma}
\begin{proof}
To prove (a), we note that if $m\in S^0$, then for all $N\in\mathbb{N}$ and $\phi\in C^\infty_c(\R_+)$, $\phi\, m(t\cdot)\in L^2_N(\R)$ uniformly in $t>0$, which satisfies the ``mixed H\"{o}rmander-Mikhlin condition'' (see \cite[Proposition 4.2]{Va}) of order $(s_0,s_\infty)$ for arbitrary $s_0>3/2$ and $s_\infty>\max\{3/2,(n+1)/2\}$. So by \cite[Proposition 4.2]{Va}, we obtain (a).

To prove (b), we let $\{\beta_j\}_{j=0}^\infty$ be an inhomogeneous partition of unity, i.e., there exists $\beta\in C^\infty(\R)$ with $\supp \beta\subset [1/4,4]$ satisfying $\sum_{j=-\infty}^\infty \beta(2^{-j}\lambda)\equiv 1$ for all $\lambda \in \R_+$, such that
\begin{align*}
\beta_0:=\sum_{j=-\infty}^0\beta(2^{-j}\cdot)\quad \text{and}\quad \beta_j=\beta(2^{-j}\cdot)\text{ for }j=1,2,\dots.
\end{align*}
We set $m_j:=m\,\beta_j$ and $\widetilde{m}_j(\lambda):=m_j(\lambda)\exp(2^{-2j}\lambda^2)$ for $j\in\mathbb{N}$. Note that
\begin{align*}
\nabla m(\sqrt{L})=\sum_{j=0}^\infty \nabla m_j(\sqrt{L})=\sum_{j=0}^\infty \nabla \widetilde{m}_j(\sqrt{L})\exp(-2^{-2j}L).
\end{align*}
On the one hand, following the proof of \cite[Proposition 5.2]{MuVa}, if $m\in S^{-1}$, then for all $\varepsilon\geq 0$, the Schwartz kernels $\mathbf{K}_{m,j}(\cdot,\cdot)$ and $\mathbf{K}_{\widetilde{m},j}(\cdot,\cdot)$ with respect to the right Haar measure of the operators $\nabla m_j(\sqrt{L})$ and $\nabla\widetilde{m}_j(\sqrt{L})$, respectively, satisfy
\begin{align}\label{eq:2.16.1}
\sup_{j=0,1,\dots}\sup_{h\in G}\int_G|\mathbf{K}_{m,j}(g,h)|_\g\,(1+2^jd(g,h))^\varepsilon\,\d\rho(g)<\infty,
\end{align}
\begin{align}\label{e2.17}
\sup_{j=0,1,\dots}\sup_{h\in G}\int_G|\mathbf{K}_{\widetilde{m},j}(g,h)|_\g\,\d\rho(g)<\infty.
\end{align}
On the other hand, for $j=0,1,\dots$ and $h,h'\in G$, by \cite[Lemma 6.3]{HeSt} and \eqref{e2.17}, we have
\begin{align*}
&\int_G|\mathbf{K}_{m,j}(g,h)-\mathbf{K}_{m,j}(g,h')|_\g\,\d\rho(g)\\
&\qquad\leq \int_G\left(\int_G|\mathbf{K}_{\widetilde{m},j}(g,\widetilde{h})|_\g\,\d\rho(g)\right)|H_j(\widetilde{h},h)-H_j(\widetilde{h},h')|\,\d\rho(\widetilde{h})\\
&\qquad\leq C 2^jd(h,h'),
\end{align*}
where $H_j(\cdot,\cdot)$ is the Schwartz kernel with respect to the right Haar measure $\d\rho$ of the operator $\exp(-2^{-2j}L)$. Therefore, the two conditions in \cite[Theorem 1.2]{HeSt} hold. By \cite[Theorem 4.1]{Va}, (b) is proved.
\end{proof}
\medskip


\section{Explicit asymptotic expansions of spherical functions}\label{sec4}
Let $\d\sigma_t$ be the normalized spherical measure of radius $t$ on $G$, i.e.
\begin{align*}
\int_G f\,\d\sigma_t={1\over\nu_n}\int_{\SS^n}f(g(t,\omega))\,\d\omega,
\end{align*}
where by $g(r,\omega)$ we denote the tuple $(x(r,\omega),y(r,\omega))$ for $x(r,\omega)$ and $y(r,\omega)$ as in \eqref{pol}. It follows that the spherical function $\varphi_\lambda(t)$ given by \eqref{e2.10} is the spherical Fourier transform of $\d\sigma_t$, i.e.,
\begin{align*}
\F(\d\sigma_t)(\lambda)=\varphi_\lambda(t),
\end{align*}
which implies
\begin{align*}
\varphi_{\sqrt{\Delta_n}}(t)(f)(x,y)=f*\d\sigma_t(x,y)={1\over \nu_n}\int_{\SS^n}f((x,y)\cdot g(t,\omega))\,\d \omega
\end{align*}
according to Lemma \ref{lem3.1}. By \eqref{e3.3}, we have
\begin{align}
\varphi_{\sqrt{L}}(t)(f)(x,y)&=\delta^{1/2}(x,y)\,\varphi_{\sqrt{\Delta_n}}(t)(\delta^{-1/2}f)(x,y)\notag\\
&={1\over\nu_n}\int_{\SS^n}\delta^{-1/2}(g(t,\omega))\, f((x,y)\cdot g(t,\omega))\,\d \omega.\label{e4.1}
\end{align}
\smallskip

We now give the explicit asymptotic expansions of $\varphi_\lambda(t)$ and $\varphi_\lambda'(t)$ for large $\lambda$.
\begin{proposition}\label{prop4.1}
Suppose $t>0$ and $\lambda\geq \max\{1,t^{-1}\}$.
\begin{asparaenum}[\rm (a)]
\item For $\varphi_\lambda(t)$ we have
\begin{align*}
\varphi_\lambda(t)={2^{n/2}\over\sqrt{\uppi}}\Gamma\!\left({n+1\over 2}\right) (\sinh t)^{-n/2} \lambda^{-n/2}\cos(t\lambda-n\uppi/4)+\mathfrak{r}_t(\lambda),
\end{align*}
where
\begin{align*}
\mathfrak{r}_t(\lambda)=\begin{cases}
\e^{\i t\lambda}a_t(t\lambda)+\e^{-\i t\lambda}a_t(-t\lambda)&\mbox{if }0<t\leq 1 \mbox{ and }t\lambda\geq 1,\\
\e^{-nt/2}\left(\e^{\i t\lambda}a_t(\lambda)+\e^{-\i t\lambda}a_t(-\lambda)\right)&\mbox{if }t>1 \mbox{ and }\lambda\geq 1
\end{cases}
\end{align*}
for some $a_t\in S^{-1-n/2}$ uniformly in $t>0$.
\item For $\varphi_\lambda'(t)$ we have
\begin{align*}
\varphi_\lambda'(t)=-{2^{n/2}\over\sqrt{\uppi}}\Gamma\!\left({n+1\over 2}\right) (\sinh t)^{-n/2} \lambda^{1-n/2}\sin(t\lambda-n\uppi/4)+\mathfrak{s}_t(\lambda),
\end{align*}
where
\begin{align*}
\mathfrak{s}_t(\lambda)=\begin{cases}
t^{-1}\left(\e^{\i t\lambda}b_t(t\lambda)+\e^{-\i t\lambda}b_t(-t\lambda)\right)&\mbox{if }0<t\leq 1 \mbox{ and }t\lambda\geq 1,\\
\e^{-nt/2}\left(\e^{\i t\lambda}b_t(\lambda)+\e^{-\i t\lambda}b_t(-\lambda)\right)&\mbox{if }t>1 \mbox{ and }\lambda\geq 1
\end{cases}
\end{align*}
for some $b_t\in S^{-n/2}$ uniformly in $t>0$.
\end{asparaenum}
\end{proposition}
\begin{proof}
We first show (a). We consider two cases separately: $0<t\leq 1$ and $t>1$.

\noindent\textbf{Case 1: $0<t\leq 1$.} In view of \cite[Theorem 2.1]{StTo},
\begin{align}
\varphi_\lambda(t)&={\Gamma((n+1)/2)\over \Gamma(n/2)}\left({t\over \sinh t}\right)^{n/2}\label{e3.0.6}\\
&\qquad \times\sum_{m=0}^\infty 2^{(n-1)/2+m}\Gamma\!\left({n\over 2}+m\right)a_m(t)\,J_{(n-1)/2+m}(\lambda t)\,(\lambda t)^{(1-n)/2-m}t^{2m},\notag
\end{align}
where $J_\mu$ denotes the classical Bessel function of order $\mu$, and the coefficients $a_m$ satisfy
\begin{align*}
a_0\equiv 1\quad \mbox{and}\quad |a_m(t)|\leq C\, T^{-m}
\end{align*}
for all $m=1,2,\dots$ and for some constant $T>1$. It follows from \cite[p. 356]{St} that
\begin{align*}
J_\mu(r)=\left(\uppi r\over 2\right)^{-1/2}\cos(r-\uppi\mu/2-\uppi/4)+\e^{\i r}s_\mu(r)+\e^{-\i r}s_\mu(-r)
\end{align*}
as $r\to\infty$, where $s_\mu$ is a symbol of order $-3/2$. Hence, \eqref{e3.0.6} can be rewritten as
\begin{align*}
\varphi_\lambda(t)={2^{n/2}\over \sqrt{\uppi}}\Gamma\!\left({n+1\over 2}\right) (\sinh t)^{-n/2} \lambda^{-n/2}\cos(t\lambda-n\uppi/4)+\e^{\i t\lambda}a_t(t\lambda)+\e^{-\i t\lambda}a_t(-t\lambda)
\end{align*}
for $0<t\leq 1$ and $t\lambda\geq 1$ and for some symbol $a_t\in S^{-n/2-1}$ uniformly in $0<t\leq 1$. The proof of \textbf{Case 1} is therefore concluded.

\noindent\textbf{Case 2: $t>1$.} In view of \cite[(8) and (10)]{AnPiVa}, we have the following Harish-Chandra expansion
\begin{align}\label{e3.2}
\varphi_\lambda(t)=(2\sinh t)^{-n/2}\left(\e^{\i\lambda t}\bc(\lambda)\sum_{k=0}^\infty \Gamma_k(\lambda)\,\e^{-2kt}+\e^{-\i\lambda t}\bc(-\lambda)\sum_{k=0}^\infty \Gamma_k(-\lambda)\,\e^{-2k t}\right),
\end{align}
where the $\bc$-function is given by \eqref{e2.11}. The coefficients satisfy $\Gamma_0\equiv 1$, and for each $k=1,2,\dots$ and $\ell=0,1,\dots$,
\begin{align*}
|\partial^\ell_\lambda \Gamma_k(\lambda)|\leq C_\ell k^\nu (1+|\lambda|)^{-\ell-1},\quad\mbox{for all }\lambda\in\R\setminus \{0\},
\end{align*}
for some constant $\nu>0$, see \cite[Lemma 2.1]{AnPiVa}. We rewrite \eqref{e3.2} as
\begin{align}\label{e3.0.4}
\varphi_\lambda(t)=(2\sinh t)^{-n/2}\left(\e^{\i\lambda t}\bc(\lambda)+\e^{-\i\lambda t}\bc(-\lambda)\right)+(\sinh t)^{-n/2}\left(\e^{\i\lambda t}\widetilde{a}_t(\lambda)+\e^{-\i\lambda t}\widetilde{a}_t(-\lambda)\right),
\end{align}
where
\begin{align*}
\widetilde{a}_t(\lambda):=2^{-n/2}\bc(\lambda)\sum_{k=1}^\infty \Gamma_k(\lambda)\,\e^{-2kt}
\end{align*}
is a symbol in $S^{-n/2-1}$ uniformly in $t>1$, according to the fact that
\begin{align}\label{e3.0.5}
\bc(\lambda)={2^{n-1}\Gamma((n+1)/2)\over\sqrt{\uppi}}\left((\i\lambda)^{-n/2}+O(|\lambda|^{-n/2-1})\right)\quad\mbox{as }|\lambda|\to\infty,
\end{align}
which follows from \eqref{e2.11} and \cite[p. 47, (4)]{Er}. It remains to substitute \eqref{e3.0.5} into \eqref{e3.0.4} to conclude the proof of \textbf{Case 2}.

We now show (b). It follows from \cite[p. 5616]{AnPiVa} that
\begin{align}\label{e3.0.7}
\varphi_\lambda(t)={}_2F_1\!\left({n\over 4}+\i{\lambda\over 2},{n\over 4}-\i{\lambda\over 2};{n+1\over 2};-\sinh^2t\right),
\end{align}
where ${}_2F_1$ is the Gauss hypergeometric function given by
\begin{align*}
{}_2F_1(a,b;c;z)={\Gamma(c)\over \Gamma(b)\,\Gamma(c-b)}\int_0^1\tau^{b-1}(1-\tau)^{c-b-1}(1-z\tau)^{-a}\,\d \tau.
\end{align*}
It follows that
\begin{align*}
{\partial\over \partial z}\big[{}_2F_1(a,b;c;z)\big]={ab\over c}{}_2F_1(a+1,b+1;c+1;z).
\end{align*}
Furthermore, by \cite[p. 112, (18)]{Er}, \eqref{e3.0.7} can be rewritten as
\begin{align}\label{e3.0.8}
\varphi_\lambda(t)={}_2F_1\!\left({n\over 2}+\i\lambda,{n\over 2}-\i\lambda;{n+1\over 2};-\sinh^2{t\over 2}\right).
\end{align}
Therefore, by differentiating both sides of \eqref{e3.0.8} with respect to $t$, we obtain
\begin{align}
\varphi_\lambda'(t)&=-\left[{}_2F_1\!\left({n\over 2}+\i\lambda,{n\over 2}-\i\lambda;{n+1\over 2};\cdot\right)\right]'\!\!\left(-\sinh^2{t\over 2}\right)\,\sinh{t\over 2}\,\cosh{t\over 2}\label{e3.100}\\ 
&=\left(-{n^2+4\lambda^2\over 4(n+1)}\sinh t \right){}_2F_1\!\left({n+2\over 2}-\i\lambda,{n+2\over 2}+\i\lambda;{n+3\over 2};-\sinh^2{t\over 2}\right)\notag
\end{align}
The key observation is that, by \eqref{e3.0.8}, the function ${}_2F_1((n+2)/2-\i\lambda,(n+2)/2+\i\lambda;(n+3)/2;-\sinh^2(t/2))$ is the spherical function on the $ax+b$ group $\R_+\ltimes\R^{n+2}$. Therefore, we deduce from (a) that
\begin{align}
{}&{}_2F_1\!\left({n+2\over 2}-\i\lambda,{n+2\over 2}+\i\lambda;{n+3\over 2};-\sinh^2{t\over 2}\right)\label{e3.0.10}\\
={}&{2^{n/2}(n+1)\over\sqrt{\uppi}}\Gamma\!\left({n+1\over 2}\right) (\sinh t)^{-n/2-1} \lambda^{-n/2-1}\cos(t\lambda-(n+2)\uppi/4)+\widetilde{\mathfrak{r}}_t(\lambda),\notag
\end{align}
where
\begin{align*}
\widetilde{\mathfrak{r}}_t(\lambda)=\begin{cases}
\e^{\i t\lambda}\widetilde{b}_t(t\lambda)+\e^{-\i t\lambda}\widetilde{b}_t(-t\lambda)&\mbox{if }0<t\leq 1 \mbox{ and }t\lambda\geq 1,\\
\e^{-(n+2)t/2}\left(\e^{\i t\lambda}\widetilde{b}_t(\lambda)+\e^{-\i t\lambda}\widetilde{b}_t(-\lambda)\right)&\mbox{if }t>1 \mbox{ and }\lambda\geq 1
\end{cases}
\end{align*}
for some $\widetilde{b}_t\in S^{-(n+4)/2}$ uniformly in $t>0$. It remains to substitute \eqref{e3.0.10} into \eqref{e3.100} to obtain (b).
\end{proof}
\medskip


\section{Proof of Theorem \ref{thm1.1}: the sufficiency part}\label{sec5}
In this section we show the sufficiency part of Theorem \ref{thm1.1}. For the purpose we prove the following result.
\begin{theorem}\label{thm5.1}
Let $G$ be the group $\R_+\ltimes \R^n$ which is of $ax+b$ type, and $L$ be the distinguished left invariant Laplacian on $G$. Assume $t\in\R\setminus\{0\}$.
\begin{asparaenum}[\rm (i)]
\item
Let $1<p<\infty$ and $\alpha_0=n|1/p-1/2|$ and $\alpha_1=n|1/p-1/2|-1$. The solution $u=u(t,\cdot)$ to the wave equation \eqref{e1.4} satisfies
\begin{align*} 
\|u(t,\cdot)\|_{L^p(G)}\leq C_p\,\Big((1+|t|)^{2|1/p-1/2|}\|f_0\|_{L^p_{\alpha_0}(G)}+(1+|t|)\,\|f_1\|_{L^p_{\alpha_1}(G)}\Big).
\end{align*}
\item
Let $\alpha_0=n/2$ and $\alpha_1=n/2-1$. Then we have 
\begin{align*} 
\|u(t,\cdot)\|_{L^1(G)}\leq C\,(1+|t|)\Big(\|(\Id+L)^{\alpha_0/2}f_0\|_{H^1(G)}+\|(\Id+L)^{\alpha_1/2}f_1\|_{H^1(G)}\Big).
\end{align*}
\end{asparaenum}
\end{theorem}

The following proposition plays an essential role in the proof of Theorem~\ref{thm5.1}. 
\begin{proposition}\label{prop5.2}
Let $t\in\R\setminus\{0\}$ and $m\in S^{-\alpha}$ be an even symbol.
\begin{asparaenum}[\rm (a)]
\item If $\alpha=n/2$, then
\begin{align*}
\|m(\sqrt{L})\,\cos(t\sqrt{L})\|_{H^1(G)\to L^1(G)}\leq C\,(1+|t|) 
\end{align*}
for some constant $C>0$ independent of $t$.
\item If $\alpha=n/2-1$, then
\begin{align*}
\left\|m(\sqrt{L})\,{\sin(t\sqrt{L})\over \sqrt{L}}\right\|_{H^1(G)\to L^1(G)}\leq C\,(1+|t|) 
\end{align*}
for some constant $C>0$ independent of $t$.
\end{asparaenum}
\end{proposition}
\begin{proof}
Without loss of generality we will assume that $t>0$. To prove (a), we suppose that $m\in S^{-n/2}$ is an even function, 
and consider two cases: $t>1$ and $0<t\leq 1$. 
\medskip

\noindent{\bf Case 1: $t>1$}.
Let $\eta\in C^\infty_c(\R)$ be an even function satisfying $\eta\equiv 1$ on $[-1,1]$. It follows from \cite[Theorem 8.1 (a)]{MuTh} that
\begin{align*}
\|\eta(\sqrt{L})\,m(\sqrt{L})\,\cos(t\sqrt{L})\|_{L^1(G)\to L^1(G)} \leq C\,( 1+|t|) 
\end{align*}
for some constant $C>0$ independent of $t$. Therefore it remains to show that
\begin{align}\label{e5.8}
\|(1-\eta(\sqrt{L}))\,m(\sqrt{L})\, \cos(t\sqrt{L})\|_{H^1(G)\to L^1(G)} \leq C\,( 1+|t|). 
\end{align}

To prove \eqref{e5.8}, the crucial observation in this paper is to use the following elementary identity:
\begin{align}
\cos(t\lambda)=\cos{n\uppi\over 4}\cos\!\left(t\lambda-{n\uppi\over 4}\right)-\sin{n\uppi\over 4}\sin\!\left(t\lambda-{n\uppi\over 4}\right).\label{e5.999}
\end{align}
Then we apply Proposition \ref{prop4.1} to obtain that for all $\lambda\in\supp(1-\eta)$,
\begin{align}
\cos(t\lambda) 
&=c_1\,(\sinh t)^{n/2}\lambda^{n/2}\varphi_\lambda(t) 
+c_2\,(\sinh t)^{n/2}\lambda^{n/2-1}\varphi_\lambda'(t) \nonumber\\ 
&\qquad+\underbrace{\e^{\i t\lambda}\mathfrak{a}_{1,t}(\lambda)+\e^{-\i t\lambda}\mathfrak{a}_{2,t}(\lambda)}_{\mathfrak{r}_t(\lambda)}\label{e5.9}
\end{align}
with $\mathfrak{a}_{1,t}(\lambda), \mathfrak{a}_{2,t}(\lambda)\in S^{-1}$ uniformly in $t\geq 1$.

Now we claim that
\begin{align}\label{e5.10}
\|(1-\eta(\sqrt{L}))\,m(\sqrt{L})\,\mathfrak{r}_t(\sqrt{L})\|_{L^1(G)\to L^1(G)}\leq C\, t.
\end{align}
Indeed, noting that on $\supp (1-\eta)$ we can rewrite
\begin{align*}
\mathfrak{r}_t(\lambda)=\widetilde{\mathfrak{a}}_{1,t}(\lambda)\cos(t\lambda)+\widetilde{\mathfrak{a}}_{2,t}(\lambda){\sin(t\lambda)\over \lambda},
\end{align*}
where $\widetilde{\mathfrak{a}}_{1,t}(\lambda)\in S^{-1}$, $ \widetilde{\mathfrak{a}}_{2,t}(\lambda)\in S^0$ uniformly in $t>1$. Then \eqref{e5.10} immediately follows from \cite[Theorem 8.1]{MuTh}.

Next we consider the two terms $(\sinh t)^{n/2}\lambda^{n/2}\varphi_\lambda(t) 
$ and $ (\sinh t)^{n/2}\lambda^{n/2-1}\varphi_\lambda'(t)$ in \eqref{e5.9}. 
If we set
\begin{align*}
\mathfrak{m}_{1,t}(\lambda):=(\sinh t)^{n/2}\varphi_\lambda(t)\, \mathfrak{m}_1(\lambda),\quad \mathfrak{m}_{2,t}(\lambda):=(\sinh t)^{n/2}\varphi_\lambda'(t)\, \mathfrak{m}_2(\lambda)
\end{align*}
with
\begin{align*}
\mathfrak{m}_1(\lambda):=(1-\eta(\lambda))\, m(\lambda)\,\lambda^{n/2}\in S^0\quad\mbox{and}\quad\mathfrak{m}_2(\lambda):=(1-\eta(\lambda))\, m(\lambda)\,\lambda^{n/2-1}\in S^{-1},
\end{align*}
in view of \eqref{e5.9}, the proof of \eqref{e5.8} reduces to showing that
\begin{align}\label{e5.11}
\|\mathfrak{m}_{j,t}(\sqrt{L})\|_{H^1(G)\to L^1(G)}\leq C\, t, \quad j=1,2.
\end{align}

Let us show \eqref{e5.11} for $j=1$. 
In view of Lemma \ref{thm3.2}(a), the operator $\mathfrak{m}_1(\sqrt{L})$ is bounded from $H^1(G)$ to $L^1(G)$. In addition, by \eqref{e4.1} and Lemma \ref{lem2.1}, we have
\begin{align*}
\|(\sinh t)^{n/2}\varphi_{\sqrt{L}}(t)\|_{L^1(G)\to L^1(G)}\leq C\,\e^{nt/2}\int_{\SS^n}\delta^{-1/2}(g(t,\omega))\,\d\omega\leq C\, t.
\end{align*}
Therefore,
\begin{align*}
\|\mathfrak{m}_{1,t}(\sqrt{L})\|_{H^1(G)\to L^1(G)}\leq \|(\sinh t)^{n/2}\varphi_{\sqrt{L}}(t)\|_{L^1(G)\to L^1(G)}\|\mathfrak{m}_1(\sqrt{L})\|_{H^1(G)\to L^1(G)}\leq C\, t,
\end{align*}
and \eqref{e5.11} holds for $j=1$.

To verify \eqref{e5.11} for $j=2$, we apply \eqref{e4.1} to write
\begin{align*}
{}&\mathfrak{m}_{2,t}(\sqrt{L})(f)(g)\\
={}&(\sinh t)^{n/2}\partial_t\!\left[\varphi_{\sqrt{L}}(t)\,\mathfrak{m}_2(\sqrt{L})(f)(g)\right]\\
={}&{1\over\nu_n}(\sinh t)^{n/2}\partial_t\!\left[\int_{\SS^n}\delta^{-1/2}(g(t,\omega))\,\mathfrak{m}_2(\sqrt{L})(f)(g\cdot g(t,\omega))\,\d\omega\right]\\
={}&{1\over\nu_n}\underbrace{(\sinh t)^{n/2}\int_{\SS^n}\partial_t[\delta^{-1/2}(g(t,\omega))]\,\mathfrak{m}_2(\sqrt{L})(f)(g\cdot g(t,\omega))\,\d\omega}_{T_{1,t}(f)(g)}\\
{}&+{1\over\nu_n}\underbrace{(\sinh t)^{n/2}\int_{\SS^n}\delta^{-1/2}(g(t,\omega))\langle\nabla \mathfrak{m}_2(\sqrt{L})(f)(g\cdot g(t,\omega)),\partial_t[g\cdot g(t,\omega)]\rangle_{\mathbf{g}}\,\d \omega}_{T_{2,t}(f)(g)},
\end{align*}
where $\langle\cdot,\cdot\rangle_{\mathbf{g}}$ is the Riemannian metric on $G$ and $\nabla$ is the gradient given by \eqref{e2.1}. Now by Lemma \ref{lem2.1} and Lemma \ref{thm3.2}(a),
\begin{align*}
\|T_{1,t}\|_{H^1(G)\to L^1(G)}\leq C\,\e^{nt/2}\int_{\SS^n}|\partial_t[\delta^{-1/2}(g(t,\omega))]|\,\d\omega\,\|\mathfrak{m}_2(\sqrt{L})\|_{H^1(G)\to L^1(G)}\leq C\,t.
\end{align*}
Meanwhile, since $g\cdot g(\cdot,\omega)$ is a geodesic in $G$, we have
\begin{align*}
|\partial_t[g\cdot g(t,\omega)]|_{\mathbf{g}}\equiv1\quad \mbox{for }t>0.
\end{align*}
So by Schwartz's inequality, Lemmas \ref{lem2.1} and \ref{thm3.2}(b), we can write
\begin{align*}
\|T_{2,t}(f)\|_{L^1(G)}\leq C\,\e^{nt/2}\int_{\SS^n}\delta^{-1/2}(g(t,\omega))\,\d \omega\,\||\nabla\mathfrak{m}_2(\sqrt{L})(f)|_{\mathbf{g}}\|_{L^1(G)}\leq C\,t\,\|f\|_{H^1(G)}.
\end{align*}
Combining these estimates together we obtain
\begin{align*}
\|\mathfrak{m}_{2,t}(\sqrt{L})\|_{H^1(G)\to L^1(G)}\leq C\,t,
\end{align*}
and estimate \eqref{e5.11} holds for $j=2$. From \eqref{e5.11}, we conclude the proof of (a) for $t>1$ in {\bf Case 1}.
\medskip

\noindent
{\bf Case 2: $0<t\leq 1$}. Let $\eta\in C^\infty_c(\R)$ be the cutoff function as above. We set $m_{t,0}(\lambda):=\eta(t\lambda)\,m(\lambda)\,\cos(t\lambda)$. Note that $m_{t,0}\in S^{-n/2}\subset S^0$ uniformly in $0<t\leq 1$. We apply Lemma \ref{thm3.2}(a) to obtain
\begin{align*}
\|m_{t,0}(\sqrt{L})\|_{H^1(G)\to L^1(G)}\leq C.
\end{align*}
Therefore it remains to show that
\begin{align}\label{e5.4}
\|(1-\eta(t\sqrt{L}))\,m(\sqrt{L})\,\cos(t\sqrt{L})\|_{H^1(G)\to L^1(G)}\leq C.
\end{align}

We apply \eqref{e5.999} and Proposition \ref{prop4.1} again to obtain that for all $\lambda\in\supp (1-\eta(t\cdot))$,
\begin{align}
\cos(t\lambda)&=\cos{n\uppi\over 4}\cos\!\left(t\lambda-{n\uppi\over 4}\right)-\sin{n\uppi\over 4}\sin\!\left(t\lambda-{n\uppi\over 4}\right)\notag\\
&=c_1\, (\sinh t)^{n/2}\lambda^{n/2}\varphi_\lambda(t)+c_2\,(\sinh t)^{n/2} \lambda^{n/2-1}\varphi_\lambda'(t)\nonumber \\
&\qquad+\underbrace{\e^{\i t\lambda}\mathfrak{a}_{1,t}(t\lambda)+\e^{-\i t\lambda}\mathfrak{a}_{2,t}(t\lambda)}_{\mathfrak{r}_t(\lambda)}\label{e5.5}
\end{align}
with $\mathfrak{a}_{1,t}(\lambda), \mathfrak{a}_{2,t}(\lambda)\in S^{-1}$ uniformly in $0<t\leq 1$.

Now we claim that
\begin{align}\label{e5.6}
\|(1-\eta(t\sqrt{L}))\,m(\sqrt{L})\,\mathfrak{r}_t(\sqrt{L})\|_{L^1(G)\to L^1(G)}\leq C.
\end{align}
Indeed, note that on $\supp (1-\eta(t\cdot))$ we can rewrite 
\begin{align*}
\mathfrak{r}_t(\lambda)=\widetilde{\mathfrak{a}}_{1,t}(t\lambda)\cos(t\lambda)+{\widetilde{\mathfrak{a}}_{2,t}(t\lambda)\over t}{\sin(t\lambda)\over \lambda},
\end{align*}
where $\widetilde{\mathfrak{a}}_{1,t}(\lambda)\in S^{-1}$, $ \widetilde{\mathfrak{a}}_{2,t}(\lambda)\in S^0$ uniformly in $0<t\leq 1$. Following the proof of \cite[Theorem 8.1]{MuTh}, we take 
$\beta\in C^\infty_c(\R)$ such that it vanishes on $[-1, 1]$, and 
\begin{align*}
(1-\eta(t\lambda))\,m(\lambda)\,\widetilde{\mathfrak{a}}_{1,t}(t\lambda)=\sum_{\ell=-\log t}^\infty m_{\ell,t}(2^{-\ell} \lambda),
\end{align*}
where
\begin{align*}
m_{\ell,t}(\lambda):=\beta(\lambda)\,m(2^\ell\lambda)\,\widetilde{\mathfrak{a}}_{1,t}(2^\ell t\lambda).
\end{align*}
Note that for all $N=1,2,\dots$,
\begin{align*}
\|m_{\ell,t}\|_{C^N} \leq C_N\,t^{-1} 2^{-(n/2+1)\ell}
\end{align*}
uniformly in $\ell\geq-\log t$ and $0<t\leq 1$. From \cite[(8.4)]{MuTh}, we have
\begin{align*}
\|m_{\ell,t}(2^{-\ell}\sqrt{L})\,\cos(t\sqrt{L})\|_{L^1(G)\to L^1(G)}\leq C\,t^{-1} 2^{-(n/2+1)\ell} 2^{n\ell/2}=C\,(2^\ell t)^{-1}.
\end{align*}
Hence
\begin{align*}
{}&\|(1-\eta(t\sqrt{L}))\,m(\sqrt{L})\,\widetilde{\mathfrak{a}}_{1,t}(t\sqrt{L})\,\cos(t\sqrt{L})\|_{L^1(G)\to L^1(G)}\\
\leq{}&\sum_{\ell=-\log t}^\infty \|m_{\ell,t}(2^{-\ell}\sqrt{L})\,\cos(t\sqrt{L})\|_{L^1(G)\to L^1(G)}\leq C.
\end{align*}
A similar argument as above yields 
\begin{align*}
\left\|(1-\eta(t\sqrt{L}))\,m(\sqrt{L}){\widetilde{\mathfrak{a}}_{2,t}(t\sqrt{L})\over t}{\sin(t\sqrt{L})\over \sqrt{L}}\right\|_{L^1(G)\to L^1(G)}\leq C.
\end{align*}
Thus \eqref{e5.6} is derived.

Next we consider the two terms $\lambda^{n/2}\varphi_\lambda(t)$ and $\lambda^{n/2-1}\varphi_\lambda'(t)$ in \eqref{e5.5}. If we set
\begin{align*}
\widetilde{\mathfrak{m}}_{1,t}(\lambda):=\varphi_\lambda(t)\,\mathfrak{m}_{1,t}(\lambda),\quad \widetilde{\mathfrak{m}}_{2,t}(\lambda):=\varphi'_\lambda(t)\,\mathfrak{m}_{2,t}(\lambda)
\end{align*}
with
\begin{align*}
\mathfrak{m}_{1,t}(\lambda):=(1-\eta(t\lambda))\, m(\lambda)\,\lambda^{n/2}\in S^0\quad \mbox{and}\quad \mathfrak{m}_{2,t}(\lambda):=(1-\eta(t\lambda))\, m(\lambda)\,\lambda^{n/2-1}\in S^{-1}
\end{align*}
uniformly in $0<t\leq 1$, in view of \eqref{e5.5}, the proof of \eqref{e5.4} reduces to showing that
\begin{align}\label{e5.7}
\|\widetilde{\mathfrak{m}}_{j,t}(\sqrt{L})\|_{H^1(G)\to L^1(G)}\leq C,\quad j=1,2.
\end{align} 

Let us show \eqref{e5.7} for $j=1$. In view of Lemma \ref{thm3.2}(a), the operator $\mathfrak{m}_{1,t}(\sqrt{L})$ is bounded from $H^1(G)$ to $L^1(G)$ uniformly in $0<t\leq 1$. In addition, by \eqref{e4.1} and Lemma \ref{lem2.1}, we have
\begin{align*}
\|\varphi_{\sqrt{L}}(t)\|_{L^1(G)\to L^1(G)}\leq C\int_{\SS^n}\delta^{-1/2}(g(t,\omega))\,\d\omega\leq C.
\end{align*}
Therefore,
\begin{align*}
\|\widetilde{\mathfrak{m}}_{1,t}(\sqrt{L})\|_{H^1(G)\to L^1(G)}\leq \|\varphi_{\sqrt{L}}(t)\|_{L^1(G)\to L^1(G)}\|\mathfrak{m}_{1,t}(\sqrt{L})\|_{H^1(G)\to L^1(G)}\leq C,
\end{align*}
and \eqref{e5.7} holds for $j=1$.

To verify \eqref{e5.7} for $j=2$, we apply \eqref{e4.1} to write
\begin{align*}
\widetilde{\mathfrak{m}}_{2,t}(\sqrt{L})(f)(g)={}&\partial_t\!\left[\varphi_{\sqrt{L}}(t)\,\mathfrak{m}_{2,t}(\sqrt{L})(f)(g)\right]-\varphi_{\sqrt{L}}(t)\,(\partial_t\mathfrak{m}_{2,t})(\sqrt{L})(f)(g)\\
={}&{1\over\nu_n}\partial_t\!\left[\int_{\SS^n}\delta^{-1/2}(g(t,\omega))\,\mathfrak{m}_{2,t}(\sqrt{L})(f)(g\cdot g(t,\omega))\,\d\omega\right]\\
{}&-{1\over\nu_n}\int_{\SS^n}\delta^{-1/2}(g(t,\omega))\,(\partial_t\mathfrak{m}_{2,t})(\sqrt{L})(f)(g\cdot g(t,\omega))\,\d\omega\\
={}&{1\over\nu_n}\underbrace{\int_{\SS^n}\partial_t[\delta^{-1/2}(g(t,\omega))]\,\mathfrak{m}_{2,t}(\sqrt{L})(f)(g\cdot g(t,\omega))\,\d\omega}_{T_{1,t}(f)(g)}\\
{}&+{1\over\nu_n}\underbrace{\int_{\SS^n}\delta^{-1/2}(g(t,\omega))\langle\nabla \mathfrak{m}_{2,t}(\sqrt{L})(f)(g\cdot g(t,\omega)),\partial_t[g\cdot g(t,\omega)]\rangle_{\mathbf{g}}\,\d \omega}_{T_{2,t}(f)(g)}.
\end{align*}
Now by Lemma \ref{lem2.1} and Lemma \ref{thm3.2}(a),
\begin{align*}
\|T_{1,t}\|_{H^1(G)\to L^1(G)}\leq\int_{\SS^n}|\partial_t[\delta^{-1/2}(g(t,\omega))]|\,\d\omega\,\|\mathfrak{m}_2(\sqrt{L})\|_{H^1(G)\to L^1(G)}\leq C.
\end{align*}
On the other hand, by Schwartz's inequality, Lemmas \ref{lem2.1} and \ref{thm3.2}(b), we can write
\begin{align*}
\|T_{2,t}(f)\|_{L^1(G)}\leq \int_{\SS^n}\delta^{-1/2}(g(t,\omega))\,\d \omega\,\||\nabla\mathfrak{m}_{2,t}(\sqrt{L})(f)|_{\mathbf{g}}\|_{L^1(G)}\leq C\,\|f\|_{H^1(G)}.
\end{align*}
Combining these estimates together we obtain
\begin{align*}
\|\widetilde{\mathfrak{m}}_{2,t}(\sqrt{L})\|_{H^1(G)\to L^1(G)}\leq C,
\end{align*}
and estimate \eqref{e5.7} holds for $j=2$. From \eqref{e5.7} we conclude the proof of (a) for $0<t\leq 1$ in \textbf{Case 2}. This ends the proof of (a).

The proof of (b) can be obtained by a similar argument, and we skip the details here. The proof of Proposition~\ref{prop5.2} is concluded.
\end{proof}

Next we apply the $H^1$--$L^1$-boundedness of wave propagators in Proposition~\ref{prop5.2} to show the following result.
\begin{proposition}\label{prop5.3}
Let $t\in\R\setminus\{0\}$ and $1<p<\infty$. Suppose $m_1\in S^{-n|1/p-1/2|}$ and $m_2\in S^{-n|1/p-1/2|+1}$ are two even symbols. Then we have
\begin{align*}
\|m_1(\sqrt{L})\,\cos(t\sqrt{L})\|_{L^p(G)\to L^p(G)}\leq C\,(1+|t|)^{2|1/p-1/2|}
\end{align*}
and
\begin{align*}
\|m_2(\sqrt{L})\,\sin(t\sqrt{L})/\sqrt{L}\|_{L^p(G)\to L^p(G)}\leq C\,(1+|t|).
\end{align*}
\end{proposition}
\begin{proof}
By duality, it suffices to assume $1<p\leq 2$. For $j=1,2$ we set
\begin{align*}
m^\zeta_j(\lambda):=\e^{(\zeta-2+2/p)^2}m_j(\lambda)\,(1+\lambda^2)^{n|1/p-1/2|/2-n(1-\zeta)/4}.
\end{align*}
Note that $m^\zeta_1\in S^{-n(1-\Re \zeta)/2}$ and $m^\zeta_2\in S^{-n(1-\Re \zeta)/2+1}$ uniformly in $\zeta\in \{z\in\mathbb{C}:0\leq \Re z\leq 1\}$. Define the analytic families of operators
\begin{align*}
T^\zeta_{1,t}:=m^\zeta_1(\sqrt{L})\,\cos(t\sqrt{L})\quad \mbox{and}\quad T^\zeta_{2,t}:=m^\zeta_2(\sqrt{L}){\sin(t\sqrt{L})\over\sqrt{L}}.
\end{align*}
In view of Proposition \ref{prop5.2},
\begin{align*}
\|T^\zeta_{j,t}\|_{H^1(G)\to L^1(G)}\leq C\,(1+|t|) \quad\mbox{for } j=1,2,\ \Re\zeta=0,
\end{align*}
while by the spectral theorem
\begin{align*}
\|T^\zeta_{1,t}\|_{L^2(G)\to L^2(G)}\leq C\mbox{ and } \|T^\zeta_{2,t}\|_{L^2(G)\to L^2(G)}\leq C\,(1+|t|)\quad \mbox{for } \Re\zeta=1.
\end{align*}
Therefore, by Lemma \ref{lem2.7}, we have
\begin{align*}
\|T^{2(1-1/p)}_{1,t}\|_{L^p(G)\to L^p(G)}\leq C\,(1+|t|)^{2(1/p-1/2)} \mbox{ and } \|T^{2(1-1/p)}_{2,t}\|_{L^p(G)\to L^p(G)}\leq C\,(1+|t|).
\end{align*}
But $T^{2(1-1/p)}_{1,t}=m_1(\sqrt{L})\,\cos(t\sqrt{L})$ and $T^{2(1-1/p)}_{2,t}=m_2(\sqrt{L})\,\sin(t\sqrt{L})/\sqrt{L}$, the desired result follows.
\end{proof}
Finally we show Theorem \ref{thm5.1}.
\begin{proof}[Proof of Theorem \ref{thm5.1}]
This is a direct consequence of Propositions \ref{prop5.2} and \ref{prop5.3}.
\end{proof}

\bigskip


\section{\texorpdfstring{Proof of Theorem \ref{thm1.1}: the necessity part}{Proof of Theorem \ref{thm1.1}: the necessity part}}\label{sec6}
In this section we show the necessity part of Theorem \ref{thm1.1} as well as the assertion on sharpness of the polynomial-in-time growth. In the following, we fix a $t>0$.

\begin{proof}[Proof of the necessity part of Theorem \ref{thm1.1}]
We assume that $f_1=0$, since the proof can be obtained by a similar argument when $f_1\not=0$, and we skip the detail here. 

We consider three cases respectively:
\medskip

\noindent{\bf Case 1: $p=2$}. In this case, the necessary condition $\alpha_0\geq 0$ of Theorem \ref{thm1.1} follows from the spectral theorem. 
\medskip

\noindent{\bf Case 2: $p=4n$}. In this case, we will show that $\alpha_0\geq n|1/p-1/2|=n/2-1/4$. To show it, we let $\eta\in C^\infty(\R)$ be an even function such that $0\leq \eta \leq 1$, $\eta\equiv 0$ on $[-1,1]$ and $\eta\equiv 1$ on $(-\infty,-2]\cup [2,\infty)$. It follows from \cite[Theorem 8.1]{MuTh} that for all $1\leq q\leq \infty$ and all $\alpha_0\in\R$, the operator $\cos (t\sqrt{L})\,(1-\eta(\sqrt{L}))\,(\Id+ L)^{-\alpha_0/2}$ is bounded on $L^q(G)$. To consider the operator $\cos (t\sqrt{L})\,\eta(\sqrt{L})\,(\Id+ L)^{-\alpha_0/2}$, it follows from Lemma \ref{thm3.2}(a) that its $L^{4n}$-boundedness is equivalent to that of the operator $\cos (t\sqrt{L})\,\eta(\sqrt{L})\,m(\sqrt{L})$ whenever $m\in S^{-\alpha_0}$ with $m(\lambda)\sim \lambda^{-\alpha_0}$ as $\lambda\to \infty$. So, it reduces to showing that if the operator $\cos (t\sqrt{L})\,\eta(\sqrt{L})\,m(\sqrt{L})$ is bounded on $L^{4n}(G)$ for some $m$ as above, then we must have $\alpha_0\geq n/2-1/4$.

To show it, we construct a family of functions $\{f^{(\varepsilon)}_t\}$ as follows:
\begin{align*}
f^{(\varepsilon)}_t(x,y):=\delta^{1/2}(x,y)\,\indicator_{E^{(\varepsilon)}_t}(x,y),
\end{align*}
where $E^{(\varepsilon)}_t:=\{(x,y)\in G:|R(x,y)-t-2\varepsilon|<\varepsilon\}$ for $\varepsilon$ sufficiently small. Since the modular function $\delta$ is a continuous function, obviously
\begin{align}\label{e5.0}
\|f^{(\varepsilon)}_t\|_{L^{4n}(G)}\leq C(t)\,\varepsilon^{1/(4n)}.
\end{align}

Next we need to estimate 
\begin{align}\label{e6.00001}
\|\cos (t\sqrt{L})\,\eta(\sqrt{L})\,m(\sqrt{L})(f^{(\varepsilon)}_t)\|_{L^{4n}(G)}.
\end{align}
We denote by $\kappa_t$ the convolution kernel of $\cos (t\sqrt{\Delta_n})\,\eta(\sqrt{\Delta_n})\,m(\sqrt{\Delta_n})$, where $\Delta_n$ is the shifted Laplacian defined in Subsection \ref{sec3}. In view of Lemma \ref{lem3.1} and Proposition \ref{prop4.1} (a), we have
\begin{align}\label{e6.0001}
{}&\kappa_t(x,y)\\
={}&{2^{n-2}\Gamma((n+1)/2)\over\uppi^{(n+3)/2}}\int_0^\infty\cos(t\lambda)\,\eta(\lambda)\, m(\lambda)\,\varphi_\lambda(R(x,y))\,|\bc(\lambda)|^{-2}\,\d\lambda \nonumber\\
={}&C_n\underbrace{(\sinh R(x,y))^{-n/2}\int_0^\infty\cos(t\lambda)\,\eta(\lambda)\, m(\lambda)\,\lambda^{-n/2} \cos(\lambda R(x,y)-n\uppi/4)\,|\bc(\lambda)|^{-2}\,\d\lambda}_{\widetilde{\kappa}_t(x,y)}\nonumber\\
{}&\quad+{2^{n-2}\Gamma((n+1)/2)\over\uppi^{(n+3)/2}}\underbrace{ \int_0^\infty\cos(t\lambda)\,\eta(\lambda)\, m(\lambda)\, \mathfrak{r}_{R(x,y)}(\lambda)\,|\bc(\lambda)|^{-2}\,\d\lambda}_{\mathfrak{R}_t(x,y)}, \nonumber 
\end{align}
where $\bc$ is the Harish-Chandra function given by \eqref{e2.11} and $\mathfrak{r}_t(\lambda)$ is as in Proposition \ref{prop4.1} (a). By \eqref{e2.11} and the basic properties of gamma function \cite[(2), (5) and (7) on p. 3]{Er}, we have that $|\bc(\lambda)|^{-2}\sim \lambda^n$ as $\lambda\to\infty$. We set $m(\lambda):=|\bc(\lambda)|^2\lambda^n\lambda^{-\alpha_0}$.

We first estimate $\widetilde{\kappa}_t(x,y)$ in \eqref{e6.0001}. Note that
\begin{align}\label{e6.01}
\widetilde{\kappa}_t(x,y)=(\sinh R(x,y))^{-n/2}\int_0^\infty\cos(t\lambda)\,\eta(\lambda)\, \lambda^{n/2-\alpha_0} \cos(\lambda R(x,y)-n\uppi/4)\,\d\lambda.
\end{align}
In view of \cite[Theorem 2.4.6]{Gr2}, for $R>0$ and $\gamma\ne 2K+1$ for all $K\in\mathbb{Z}$,
\begin{align}\label{e6.001}
\int_0^\infty \cos(\lambda R)\,\lambda^\gamma\,\d\lambda
={\Gamma((\gamma+1)/2)\over \Gamma(-\gamma/2)}2^{\gamma}\sqrt{\uppi} R^{-1-\gamma}.
\end{align}
Taking derivatives with respect to $R$ on both sides of \eqref{e6.001} and then replacing $\gamma +1$ by $\gamma$, in view of \cite[(6) and (7) on p. 3]{Er} we obtain that for $\gamma\ne 2K$ for all $K\in\mathbb{Z}$,
\begin{align*}
\int_0^\infty \sin(\lambda R)\,\lambda^{\gamma}\,\d\lambda
=-{\Gamma((\gamma+1)/2)\over \Gamma(-\gamma/2)}\cot\!\left({\uppi\gamma\over 2}\right)2^\gamma\sqrt{\uppi} \,R^{-1-\gamma}.
\end{align*}
Since
\begin{align*}
\cos(\lambda R-n\uppi/4)\,\cos(t\lambda)={1\over 2}\big(\cos((R-t)\lambda-n\uppi/4)+\cos((R+t)\lambda-n\uppi/4)\big),
\end{align*}
for all $\gamma\in (1/8,3/8)$ we have
\begin{align*}
\int_0^\infty \cos(t\lambda) \, \lambda^\gamma \cos(\lambda R-n\uppi/4)\,\d \lambda=C_\gamma\, \big((R-t)^{-\gamma-1}+(R+t)^{-\gamma-1}\big)
\end{align*}
with $C_\gamma\ne 0$ and $(R+t)^{-\gamma-1}\leq C$. In addition, for $\gamma$ as above, it is obvious that
\begin{align*}
\int_0^\infty \cos(t\lambda) \,(1-\eta(\lambda))\, \lambda^\gamma \cos(\lambda R-n\uppi/4)\,\d \lambda\leq C.
\end{align*}
The above estimates show that for all $\gamma\in (1/8,3/8)$, 
\begin{align}\label{e6.1}
\int_0^\infty \cos(t\lambda) \,\eta(\lambda)\, \lambda^\gamma\cos(\lambda R-n\uppi/4)\,\d \lambda= C_\gamma \,(R-t)^{-\gamma-1}+O(1)
\end{align} 
with $C_\gamma\ne 0$, as $R\to t^+$. We substitute \eqref{e6.1} into \eqref{e6.01} to obtain that for $n/2-3/8<\alpha_0<n/2-1/8,$
\begin{align}\label{e6.2}
\widetilde{\kappa}_t(x,y)=(\sinh R(x,y))^{-n/2}\big(C_{\alpha_0}\, (R(x,y)-t)^{\alpha_0-n/2-1}+O(1)\big)
\end{align}
with $C_{\alpha_0}\ne 0$, as $R(x,y)\to t^+$. 

Now we estimate $\mathfrak{R}_t(x,y)$ in \eqref{e6.0001}. For $n/2-3/8<\alpha_0<n/2-1/8$, we apply the classical estimates on Fourier transform of symbols \cite[p. 241]{St} to obtain that
\begin{align}\label{e6.3}
|\mathfrak{R}_t(x,y)|\leq C(R(x,y))\, (R(x,y)-t)^{\alpha_0-n/2},
\end{align}
as $R(x,y)\to t^+$, where $C(\cdot)$ is some continuous function.

Now we apply \eqref{e6.2} and \eqref{e6.3} to estimate $\|\cos (t\sqrt{L})\,\eta(\sqrt{L})\,m(\sqrt{L})(f^{(\varepsilon)}_t)\|_{L^{4n}(G)}$ in \eqref{e6.00001}. Suppose $(x,y)\in B_{\varepsilon/100}(e)$. Then for $(x', y')\in \supp f^{(\varepsilon)}_t=E^{(\varepsilon)}_t$, $d((x,y), (x', y'))-t\sim \varepsilon$. It follows from \eqref{e3.3}, \eqref{e6.2} and \eqref{e6.3} that for $n/2-3/8<\alpha_0<n/2-1/8$, we have
\begin{align*}
&\hspace{-0.6cm}|\cos (t\sqrt{L})\,\eta(\sqrt{L})\,m(\sqrt{L})(f^{(\varepsilon)}_t)(x,y)|\\
={}&\delta^{1/2}(x,y)\,|\cos (t\sqrt{\Delta_n})\,\eta(\sqrt{\Delta_n})\,m(\sqrt{\Delta_n})(\indicator_{E^{(\varepsilon)}_t})(x,y)|\\
\geq{}&\delta^{1/2}(x,y)\Big(|C_n|\,|\indicator_{E^{(\varepsilon)}_t}*\widetilde{\kappa}_t(x,y)|-|\indicator_{E^{(\varepsilon)}_t}*\mathfrak{R}_t(x,y)|\Big)\\
\geq{}&\widetilde{C}(t) \int_{t+c_1\varepsilon}^{t+c_2\varepsilon}(r-t)^{\alpha_0-n/2-1}\,\d r-C'(t) \int_{t+c'_1\varepsilon}^{t+c'_2\varepsilon}\big((r-t)^{\alpha_0-n/2}+O(1)\big)\,\d r\\
\geq{}&C''(t)\, \varepsilon^{\alpha_0-n/2},
\end{align*}
which implies
\begin{align*}
\|\cos (t\sqrt{L})\,\eta(\sqrt{L})\,m(\sqrt{L})(f^{(\varepsilon)}_t)\|_{L^{4n}(G)}&\geq\|\cos (t\sqrt{L})\,\eta(\sqrt{L})\,m(\sqrt{L})(f^{(\varepsilon)}_t)\|_{L^{4n}(B_{\varepsilon/100}(e))}\\
&\geq C''(t)\, \varepsilon^{\alpha_0-n/2} \varepsilon^{(n+1)/(4n)}.
\end{align*}
Combining this estimate with \eqref{e5.0}, for $\alpha_0$ as above, if the operator $\cos (t\sqrt{L})\,\eta(\sqrt{L})\,m(\sqrt{L})$ is bounded on $L^{4n}(G)$, then we must have
\begin{align*}
\varepsilon^{\alpha_0-n/2} \varepsilon^{(n+1)/(4n)}\leq \varepsilon^{1/(4n)},
\end{align*}
which is exactly $n/2-1/8>\alpha_0\geq n/2-1/4$.
On the other hand, if $\alpha_0\geq n/2-1/8$, then 
$\alpha_0\geq n/2-1/4$ holds.
This finishes the proof of {\bf Case 2} for $p=4n$.
\medskip

\noindent{\bf Case 3: $1\leq p<\infty$}. From {\bf Case 2}, we see that by duality, the condition $\alpha_0\geq n/2-1/4$ is also necessary for $\cos(t\sqrt{L})\,(\Id+L)^{-\alpha_0/2}$ to be bounded on $L^{4n/(4n-1)}(G)$. Now we show the condition $\alpha_0\geq n/2$ is necessary for $\cos (t\sqrt{L})\,(\Id+ L)^{-\alpha_0/2}$ to be bounded from $H^1(G)$ to $L^1(G)$, which is (ii). If this is not the case, the interpolation argument in the proof of Proposition \ref{prop5.3} and the necessity on $L^2$-boundedness in \textbf{Case 1} would yield a contradiction to the necessity on $L^{4n/(4n-1)}$-boundedness. 

Similarly, the necessity on $L^p$-boundedness ($1<p<2$) follows from that on $H^1\to L^1$-, $L^{4n/(4n-1)}$- and $L^2$-boundedness, and the necessity on $L^p$-boundedness $(2<p<\infty)$ follows from duality. This finishes the proof of {\bf Case 3} for $1\leq p<\infty$, and then the proof of the necessity part of Theorem \ref{thm1.1} is completed.
\end{proof}
\medskip

Finally, let us show the assertion on sharpness of the polynomial-in-time growth $1+|t|$ in Theorem \ref{thm1.1}. We will show a somewhat stronger result instead.
\begin{proposition}\label{prop5.1}
Suppose $1\leq p\leq\infty$ and $t>0$. Then for all $\alpha\in\R$
\begin{align*}
\left\|{\sin(t\sqrt{L})\over\sqrt{L}}\right\|_{L^p_\alpha(G)\to L^p(G)}\geq Ct.
\end{align*}
\end{proposition}
\begin{proof}
To simplify calculations, we let $s=\log x$ and adopt $(s,y)$-coordinates. This time
\begin{align*}
\d\rho(s,y)=\d s\d y,\qquad L=-X^2-\sum_{j=1}^nY_j^2=-\partial_s^2-\e^{2s}\Delta_y
\end{align*}
for $\Delta_y$ being the standard Laplacian on $\R^n$, while the Riemannian metric \eqref{riem} becomes
\begin{align}\label{riem-sy}
\g=\d s^2+\e^{-2s}\d y^2.
\end{align}

We let $\eta\in C^\infty_c([-2,2])$ with $0\leq \eta\leq 1$ and $\eta\equiv 1$ on $[-1,1]$, and define
\begin{align*}
q_t(s,y):=\eta\!\left({s\over 2t+1}+3\right)\prod_{j=1}^n\eta(y_j).
\end{align*}
Then
\begin{align*}
\supp q_t\subset \widetilde{\mathcal{T}}_t:=\{-10t-5\leq s\leq -2t-1,\, |y_j|\leq 2\}
\end{align*}
and
\begin{align*}
q_t\equiv 1\quad\text{on }\mathcal{T}_t:=\{-8t-4\leq s\leq -4t-2,\, |y_j|\leq 1\}.
\end{align*}
\smallskip

We first claim that for $\widetilde{\alpha}\in\R$,
\begin{align}\label{eq:prop5.1claim}
\|(\Id+L)^{\widetilde{\alpha}}(q_t)\|_{L^p(G)}\leq C\, (1+t)^{1/p}.
\end{align}
To see this, let $m=\max\{\lceil\widetilde{\alpha}\rceil,0\}$. Then
\begin{align}\label{eq:6.2}
\|(\Id+L)^{\widetilde{\alpha}}(q_t)\|_{L^p(G)}\leq \underbrace{\|(\Id+L)^{\widetilde{\alpha}-m}\|_{L^p(G)\to L^p(G)}}_{\mathbf I}\underbrace{\|(\Id+L)^m(q_t)\|_{L^p(G)}}_{\mathbf J}.
\end{align}
Since $(\Id+L)^m$ is a linear combination of the following differential operators:
\begin{align*}
\e^{2ks}\partial_s^\ell\Delta_y^k,\qquad 2k+\ell\leq 2m,
\end{align*}
and on $\supp q_t\subset \widetilde{\mathcal{T}}_t$
\begin{align*}
\e^{2ks}\partial_s^\ell\Delta_y^k(q_t)\leq C_{k,\ell},
\end{align*}
we have
\begin{align*}
\mathbf J\leq C_m\,(\rho(\widetilde{\mathcal{T}}_t))^{1/p}\leq C\, (1+t)^{1/p}.
\end{align*}

On the other hand, if $\widetilde{\alpha}=0,1,2,\dots$ we have $\mathbf I=1$; otherwise by \cite[Propositions 4.1, 5.2 and 5.7]{MuTh} and classical estimates on Fourier transforms of symbols \cite[p. 241]{St}, $(\Id+L)^{\widetilde{\alpha}-m}$ has an $L^1$-integrable convolution kernel, whence $\mathbf I\leq C$. In view of \eqref{eq:6.2}, we obtain \eqref{eq:prop5.1claim}.
\smallskip

With \eqref{eq:prop5.1claim}, to show Proposition \ref{prop5.1} it suffices to establish
\begin{align}\label{eq:prop5.1final}
\left\|{\sin(t\sqrt{L})\over\sqrt{L}}(q_t)\right\|_{L^p(G)}\geq Ct\,(1+t)^{1/p}.
\end{align}
To see this, by the argument showing \cite[Theorem 2.4.2]{Sog14}, we know that the wave equation \eqref{e1.4} has \textit{unit speed propagation}. In other words, for
$$
(s,y)\in S_t:=\{-7t-4\leq s\leq -5t-2,\, |y_j|\leq 1/2\},
$$
the functions
$$
{\sin(t\sqrt{L})\over\sqrt{L}}(q_t)(s,y)
$$
depend only on $q_t$ restricted to the $t$-neighborhood $\mathcal N_t(S_t)$ of $S_t$.

We claim that $\mathcal N_t(S_t) \subset \mathcal T_t$. Indeed, suppose $(s,y)\in S_t$, and let $(s(\tau),y(\tau))$ be any arc-lengthed geodesic ray starting from $(s,y)$. By \eqref{riem-sy}, this means that
\begin{align*}
|s'(\tau)|^2+\e^{-2s(\tau)}\sum_{j=1}^n|y_j'(\tau)|^2\equiv 1.
\end{align*}
Then for $0\leq t_0\leq t$,
\begin{align}\label{e6.3.0}
|s(t_0)-s|\leq \int_0^{t_0}|s'(\tau)|\,\d\tau\leq t,
\end{align}
and in view of \eqref{e6.3.0} and the fact that $s\leq -5t-2$,
\begin{align*}
|y_j(t_0)-y_j|\leq \int_0^{t_0} |y_j'(\tau)|\,\d \tau\leq t\sup_{0\leq\tau\leq t}\e^{s(\tau)}\leq t\,\e^{-4t-2}\leq {\e^{-3}\over 4}.
\end{align*}
Hence for $0\leq \tau\leq t$,
\begin{align*}
-8t-4\leq s(\tau)\leq -4t-2,\qquad |y_j(\tau)|\leq 1,
\end{align*}
whence $(s(\tau),y(\tau))\in \mathcal{T}_t$. The claim is verified.

To show \eqref{eq:prop5.1final}, we note that $u(t,\cdot)\equiv t$ is the solution to the Cauchy problem
\begin{align*}
{\partial^2u\over\partial t^2}+Lu=0,\quad u|_{t=0}\equiv 0,\quad\left.{\partial u\over\partial t}\right|_{t=0}\equiv1.
\end{align*}
As a consequence, by the unit speed propagation property and uniqueness of the solution (cf. \cite[pp. 12--13]{Sog14}), one has
\begin{align*}
{\sin(t\sqrt{L})\over \sqrt{L}}(q_t)\equiv t \qquad\text{on }S_t.
\end{align*}
Hence
\begin{align*}
\left\|{\sin(t\sqrt{L})\over \sqrt{L}}(q_t)\right\|_{L^p(G)}\geq {t\over 2}\,(\rho(S_t))^{1/p}\geq Ct\,(1+t)^{1/p}.
\end{align*}
This proves \eqref{eq:prop5.1final} and hence Proposition \ref{prop5.1}.
\end{proof}
\medskip


\smallskip\noindent
\textbf{Acknowledgements.} The authors were supported by National Key R$\&$D Program of China 2022YFA1005700 and 
by NNSF of China (No. 12571111). The authors would like to thank Peng Chen, Xi Chen, Naijia Liu, Detlef M\"{u}ller, Adam Sikora and Hong-Wei Zhang for helpful discussions and suggestions.


\smallskip
\bibliographystyle{plain}

\end{document}